\documentclass[a4paper,oneside,reqno]{amsart}
\usepackage{amsopn} 
\usepackage{amsthm}
\usepackage{amstext}
\usepackage{amssymb}
\usepackage[british]{babel}
\usepackage[UKenglish]{isodate}
\usepackage{graphicx}
\usepackage{esint}
\usepackage{tikz}
\usepackage{textcomp}
\usepackage{xfrac}
\usepackage{color}
\usepackage{ifthen}

\numberwithin{equation}{section}
\numberwithin{figure}{section}

\definecolor{grey}{rgb}{0.7,0.7,0.7}
\newcommand{\grey}{\color{grey}}
\newcommand{\black}{\color{black}}

\newtheorem{theorem}{Theorem}[section]
\newtheorem{proposition}[theorem]{Proposition}
\newtheorem{lemma}[theorem]{Lemma}

\theoremstyle{remark}
\newtheorem{rem}[theorem]{Remark}
\newenvironment{remark}{\begin{rem}}{\hfill$\lozenge$
\end{rem}}

\DeclareMathOperator\Real{Re}
\DeclareMathOperator\Imag{Im}
\renewcommand{\Re}{\Real}
\renewcommand{\Im}{\Imag}
\DeclareMathOperator\dom{dom}

\DeclareMathOperator\Log{Log}
\DeclareMathOperator\Li{Li}

\newcommand{\ov}{\overline}
\newcommand{\eps}{\varepsilon}
\newcommand{\rmO}{\mathrm{O}}
\newcommand{\rmo}{\mathrm{o}}
\newcommand{\rd}{\mathrm{d}}
\newcommand{\mr}{\mathring}

\newcommand{\LMs}[1]{L_{\mathrm{M},#1}}
\newcommand{\LMsN}[2]{L_{\mathrm{M},#1,#2}}
\newcommand{\lns}{l}
\newcommand{\lk}{\lns_k}
\newcommand{\lkarb}[1]{\lns_{#1}}
\newcommand{\lMs}[1]{\lns_{\mathrm{M},#1}}
\newcommand{\lMsN}[2]{\lns_{\mathrm{M},#1,#2}}
\newcommand{\Rsin}{R_{\rm sin}}
\newcommand{\bfe}{\mathbf{e}}
\newcommand{\FWHM}{\operatorname{FWHM}}

\newcommand{\cF}{\mathcal F}

\newcommand{\CC}{\mathbb C}
\newcommand{\NN}{\mathbb N}

\newcommand{\RR}{\mathbb R}
\newcommand{\ZZ}{\mathbb Z}

\newcommand{\defeq}{\mathrel{\mathop:}=}

\newcommand{\eqdef}{=\mathrel{\mathop:}}

\newcommand{\void}[2]{\ifthenelse{#1=0}{}{\ifthenelse{#1=1}{\grey #2 \black}{#2}}}

\begin{document}


\title[Path Laplacian operators and superdiffusive processes on graphs. II]{Path
Laplacian operators and superdiffusive processes on graphs.
II. Two-dimensional lattice}

\author{Ernesto Estrada}
\address{Department of Mathematics and Statistics, University of Strathclyde, \newline
26 Richmond Street, Glasgow G1 1XH, UK}
\email{ernesto.estrada@strath.ac.uk }
\author{Ehsan Hameed}
\address{Department of Mathematics and Statistics, University of Strathclyde, \newline
26 Richmond Street, Glasgow G1 1XH, UK}
\email{ehsan.hameed@strath.ac.uk }
\author{Matthias Langer}
\address{Department of Mathematics and Statistics, University of Strathclyde, \newline
26 Richmond Street, Glasgow G1 1XH, UK}
\email{m.langer@strath.ac.uk }
\author{Aleksandra Puchalska}
\address{Institute of Mathematics, \L{}\'od\'z University of Technology, \newline
Ul.\ W\'olcza\'nska 215, 90-924\L{}\'od\'z, Poland,}
\address{Institute of Applied Mathematics and Mechanics, University of Warsaw, \newline
Banacha 2, 02-097 Warsaw, Poland}
\email{aleksandrafalkiewicz@gmail.com}

\begin{abstract}
In this paper we consider a generalized diffusion equation on a square lattice
corresponding to Mellin transforms of the $k$-path Laplacian.
In particular, we prove that superdiffusion occurs when the parameter $s$
in the Mellin transform is in the interval $(2,4)$ and that normal diffusion
prevails when $s>4$.
\end{abstract}




\maketitle

\section{Introduction}

\noindent
Many physical systems are best represented by graphs $G=\left(V,E\right)$,
where the set of nodes (vertices) $V$ represents the entities of
the system and the set of edges $E$ describes the interactions between
these entities \cite{Graphs in Physics}. Among those systems we can
mention atomic and molecular ones as well as complex networks, which
include a vast range of complex systems embracing biological, social,
ecological, infrastructural and technological ones. Diffusion-like
processes, such as diffusion, reaction-diffusion, synchronization,
epidemic spreading, etc., are ubiquitous in those previously mentioned
systems \cite{Dynamics on networks}.  Apart from the normal diffusive
processes, where the mean square displacement (MSD) of the diffusive
particle scales linearly with time, there are many real-world examples
where anomalous diffusion takes place. In these anomalous diffusive
processes, MSD scales nonlinearly with time giving rise to subdiffusive
and superdiffusive processes \cite{metzler_klafter_review}.

In Part I \cite{PartI} of this series we introduced a new theoretical framework
to study superdiffusive processes on graphs.
In that work we considered transformations of the
so-called $k$-path Laplace operators $L_k$.
The latter are defined in a similar way as the standard graph Laplacian, but
they take only nodes into account whose distance is equal to $k$;
here the distance is measured as the length of the shortest path connecting two nodes.
Hence $L_k$ describes hops to nodes at distance $k$.
The above mentioned transformations of $L_k$ are combinations of
the form $\sum_{k=1}^\infty c_k L_k$ with some non-negative coefficients $c_k$.
This combination describes interactions with all nodes where different strengths
are used for nodes at different distances.  In general, one uses a sequence $c_k$
that is decreasing in $k$.
In particular, in \cite{PartI} we considered the Mellin transform of $L_k$,
which is obtained by choosing $c_k=k^{-s}$ with some positive parameter $s$.
The choice of the transformation has proved to be crucial in determining
the diffusive behaviour.  In \cite{PartI} we studied, in particular, the one-dimensional
path graph.  We proved that superdiffusion appears when a Mellin transform
of the $k$-path Laplace operators is considered with $s$ satisfying $1<s<3$, while
for $s>3$ normal diffusion is obtained; the latter occurs also if one considers
different transformations of $L_k$ like the Laplace and factorial transforms.

This new method adds new values to the already
existing ones for modelling anomalous diffusion. Among such existing
methods we should mention the use of random walks with L\'evy flights (RWLF)
\cite{Anomalous_Levy flights,Levy_flights_2,Levy_flights_3,Levy_flights_4}
and the use of the fractional diffusion equation (FDE)
\cite{fractional diffusion graphs,FDE_2,FDE_3,FDE_4}.
While the first method is easy to use for computer simulations, the
second is preferred for analytical studies.  However, there are different
types of definitions of fractional derivatives, such as the Caputo
fractional operator and the Riemann--Liouville fractional
operator \cite{Podlubny}, which then have different interpretations
and adapt differently to the different physical phenomena studied
with them (see \cite{fractional diffusion_1,fractional diffusion_2}).
The $k$-path Laplace operators allow
the derivation of analytical results as the FDE but use a unique
framework which is very similar to the one traditionally used in graph
and network theory.  It also allows an easy computational implementation
in the form of a random multi-hopper on graphs~\cite{multi-hopper}.

The goal of the current work is to study the solutions of
the generalized diffusion equation in 2D graphs. In particular, we
focus our attention on the abstract Cauchy problem in an infinite
square lattice. Square lattices are ubiquitous in many real-world
physical systems. It is frequently used to describe the spin-$1/2$
antiferromagnetic Heisenberg model in a variety of materials
\cite{square lattice_1,square_lattice 2,square_lattice 3,square_lattice 4}.
It is also the preferred model for two-dimensional (2D) gases and
optical lattices \cite{2D gases_1,2D gases_2,2D gases_3,optical lattice_1,optical lattice_2}.
Recently, square lattices of superconducting qubits have been used
for error correcting codes in quantum computers \cite{quantum computer}.
A very interesting discovery has been the experimental finding that
the native architecture of certain photosynthetic membranes have square
lattice shapes \cite{LH_1,LH_2,LH_3}. This finding is very relevant
for the current work as the existence of long-range interactions (LRI)
is well documented for light-harvesting complexes \cite{LH_LRI_1,LH_LRI_2,LH_LRI_3}.
The existence of LRI like the ones mathematically described by the
$k$-path Laplace operators considered here are well documented for
other systems previously mentioned here, such as cold atomic clouds,
helium Rydberg atoms and cold Rydberg gases \cite{LRI_1,LRI_2,LRI_3}.
Note also that anomalous diffusion has been observed for ultracold atoms
in 2D and 3D lattices \cite{schneider_etal12}.
Consequently, the study of a generalized diffusion model on square
lattices and proving the conditions for which superdiffusive behaviour
exists on them is of great theoretical importance due to the many
physical processes involved.

The main result of the current paper is contained in Theorem~\ref{th:super_diffusion},
which describes the asymptotic behaviour of the generalized diffusion equation
corresponding to the Mellin-transformed $k$-path Laplacian.
We prove that superdiffusion occurs when $2<s<4$ and that normal diffusion
prevails when $s>4$.
More precisely, we consider the time evolution of the solution of the generalized
diffusion equation with initial condition concentrated at one point.
As time $t$ tends to infinity, the spread of the solution (e.g.\ measured by
the full width at half maximum) grows like $t^\kappa$ with $\kappa=\frac{1}{2}$
when $s>4$, which is normal diffusion, and with $\kappa>\frac{1}{2}$
when $2<s<4$, which is a superdiffusive behaviour.

Let us give a brief outline of the contents of the paper:
in Section~\ref{sec:prelim} we recall results from Part~I \cite{PartI}.
In Section~\ref{sec:time_evolution} we study the solution of the generalized
diffusion equation and give an integral representation (Theorem~\ref{th:sol_evol_equ}).
Finally, in Section~\ref{sec:superdiffusion} we investigate the asymptotic
behaviour of the solution as time tends to infinity.
In particular, we formulate and prove or main result (Theorem~\ref{th:super_diffusion}).
Finally, we examine the behaviour of finite truncations, $\sum_{k=1}^N k^{-s}L_k$
of the Mellin transforms.  Although normal diffusion occurs in this case,
the diffusion speed can be made arbitrarily large
if $s\in(2,4)$ and $N$ is large enough; see Remark~\ref{re:finite_truncations}.

\section{Preliminaries}
\label{sec:prelim}

\noindent
At the beginning let us briefly recall some of the results given in the
first part \cite{PartI} of this article.
Let $\Gamma=(V,E)$ be an undirected, locally finite graph with set of vertices $V$
and set of edges $E$.
Moreover, let $d$ be the distance metric on $V$, i.e.\ let $d(v,w)$ be the
length of the shortest path from $v$ to $w$, and let $\delta_k(v)$ be the
$k$-path degree of a vertex $v\in V$:
\begin{equation}\label{pathdegr}
  \delta_k(v) = \#\{w\in V: d(v,w)=k\}.
\end{equation}
Let $\ell^{2}(V)$ be the Hilbert space of square-summable functions on $V$
with inner product
\[
  \langle f,g\rangle = \sum_{v\in V}f(v)\ov{g(v)},\qquad f,g\in\ell^{2}(V).
\]
For $k\in\NN$ we consider the $k$-path Laplacian, which is an operator
in $\ell^2(V)$ and defined by
\begin{equation}\label{eq:path_Laplacian}
  \bigl(L_{k}f\bigr)(v) \defeq \sum_{w\in V:\,d(v,w)=k}\bigl(f(v)-f(w)\bigr),
  \qquad f\in\dom(L_k),
\end{equation}
with maximal domain $\dom(L_k)$, i.e.\
\[
  \dom(L_k) = \Biggl\{f\in\ell^2(V):
  \sum_{v\in V}\bigg|\sum_{w\in V:\,d(v,w)=k}\bigl(f(v)-f(w)\bigr)\bigg|^2 < \infty\Biggr\}.
\]
The following properties were proved in the first part of the paper.

\begin{theorem}{\rm \cite[Theorem~2.2]{PartI}}
For each $k\in\NN$ the $k$-path Laplacian $L_k$ is a self-adjoined operator in $\ell^2(V)$.
Furthermore, the operator $L_k$ is bounded if and only if the
function $\delta_k: V \to \NN$ is bounded.
\end{theorem}

Now let us consider an Abstract Cauchy Problem of the form
\begin{equation}\label{Cauchy}
  u'(t)=-Lu(t), \qquad u(0)=\mathring{u},
\end{equation}
where $L$ is some operator in $\ell^2(V)$.
Similarly to the classical description of Brownian motion,
the solution to the system \eqref{Cauchy} with $L=L_k$, when rescaled properly,
converges to the normal distribution as time tends to infinity.
In order to build the model in which interaction among all vertices in a graph
that are joined by a path are taken into account, we use the
differential equation \eqref{Cauchy} with an operator $L$
given by a transformed $k$-path Laplacian operator:
\begin{equation}\label{k-pathLap}
  L = \sum_{k=1}^\infty c_k L_k
\end{equation}
with some coefficients $c_k\in \CC$.

The main goal is to examine the existence of superdiffusion in the process
described by \eqref{Cauchy} with an operator $L$ as in \eqref{k-pathLap}.
In \cite{PartI} we considered three transforms: the Laplace, the factorial
and the Mellin transforms, which differ in the rate of convergence to zero
of their coefficients.
It appeared that for the first two the probabilities of big jumps are too small
for superdiffusion to arise and a significant result happens only
for the Mellin transform.
In the current paper we therefore concentrate on the Mellin transform.
Let us recall the definition and some properties of the latter in the following
theorem from \cite{PartI}.

\begin{theorem}\label{proptrk-path}{\rm \cite[Theorem~3.1]{PartI}}
Let us consider an infinite graph $\Gamma$ which is locally finite and
such that its $k$-path degree $\delta_k$, defined in \eqref{pathdegr},
satisfies the condition
\begin{equation}\label{deltapolyn}
  \delta_{k,\max} \defeq \max\{\delta_k(v): v\in V\} \le Ck^\alpha
\end{equation}
for some $\alpha\ge0$ and $C>0$.  Then the \emph{Mellin-transformed $k$-path Laplacian}
\begin{equation}\label{deftLM}
  \LMs{s} \defeq \sum_{k=1}^\infty \frac{1}{k^s}L_k
\end{equation}
is a well-defined, bounded operator in $\ell^2(V)$ for $s\in\CC$ with $\Re s>\alpha+1$,
and the series in \eqref{deftLM} converges in the operator norm.
\end{theorem}

One can easily find examples of graphs for which \eqref{deltapolyn} is
satisfied and hence the operator $\LMs{s}$ is bounded, e.g.\ a path graph
or a square lattice where $\delta_{k,\max}$ equals $2$ and $4k$, respectively.
On the other hand, condition \eqref{deltapolyn} is violated
for the Cayley trees with degree of the non-pendant node equal to $r\in\NN$, $r\ge3$,
for which $\delta_{k,\max}=r(r-1)^{k-1}$.

\section{Existence and time evolution of the Mellin transform of
the $k$-path Laplacian on the square lattice}
\label{sec:time_evolution}

\noindent
Let us consider the square lattice, i.e.\ the graph $\Gamma=P_\infty\times P_\infty=(V,E)$
with vertices $V=\ZZ^2$ and edges connecting vertices $(i,j)$ and $(m,n)$ when
$|i-m|+|j-n|=1$.  We usually write $(u_{x,y})_{x,y\in\ZZ}$ for functions on $V$.

On $P_\infty\times P_\infty$ the $k$-path Laplacian $L_k$,
defined in \eqref{eq:path_Laplacian}, is given by
\begin{align*}
  (L_k u)_{x,y} = 4ku_{x,y} - \sum_{j=0}^{k-1}\Bigl[u_{x+k-j,y+j} + u_{x-k+j,y-j}
  + u_{x-j,y+k-j} + u_{x+j,y-k+j}\Bigr], & \\
  x,y\in\ZZ,\; u\in\ell^2(V). &
\end{align*}
Clearly, $L_k$ is a bounded operator.
For $m,n\in\ZZ$ let $\sigma_{m,n}:\ell^2(V)\to\ell^2(V)$ be the shift operator
defined by
\[
  (\sigma_{m,n}u)_{x,y} = u_{x+m,y+n},  \qquad x,y\in\ZZ.
\]
Then $L_k$ can be written as
\begin{equation}\label{Lkshift}
  L_k = 4kI - \sum_{j=0}^{k-1} \bigl[\sigma_{k-j,j} + \sigma_{-k+j,-j}
  + \sigma_{-j,k-j} + \sigma_{j,-k+j}\bigr].
\end{equation}

Let us consider the following Fourier transform, which is a unitary operator and
which is defined by
\begin{align*}
  & \cF: \ell^2(V) \to L^2\bigl([-\pi,\pi]^2\bigr), \\[1ex]
  & (\cF u)(p,q) = \frac{1}{2\pi}\sum_{x,y\in\ZZ} u_{x,y}e^{ipx}e^{iqy},
  \qquad p,q\in[-\pi,\pi],\; u\in\ell^2(V),
\end{align*}
and whose inverse given by
\[
  (\cF^{-1}f)_{x,y} = \frac{1}{2\pi}\int_{-\pi}^\pi \int_{-\pi}^\pi
  f(p,q)e^{-ipx}e^{-iqy}\rd p\,\rd q,
  \qquad x,y\in\ZZ,\; f\in L^2\bigl([-\pi,\pi]^2\bigr).
\]

Since
\begin{align*}
  (\cF\sigma_{m,n}u)(p,q) &= \frac{1}{2\pi}\sum_{x,y\in\ZZ} u_{x+m,y+n}e^{ipx}e^{iqy}
  = \frac{1}{2\pi}\sum_{x,y\in\ZZ} u_{x,y}e^{ip(x-m)}e^{iq(y-n)} \\[1ex]
  &= e^{-ipm}e^{-iqn}(\cF u)(p,q),
\end{align*}
we have
\begin{equation}\label{FshiftF}
  \bigl(\cF\sigma_{m,n}\cF^{-1}f\bigr)(p,q) = e^{-i(pm+qn)}f(p,q),
  \quad p,q\in[-\pi,\pi],\; f\in L^2\bigl([-\pi,\pi]^2\bigr).
\end{equation}
Together with \eqref{Lkshift} we obtain that $L_k$ is unitarily equivalent
to a multiplication operator; more precisely, the following lemma is true.

\begin{lemma}\label{le:lk}
With the notations from above we have
\begin{equation}\label{FLkF}
  \bigl(\cF L_k\cF^{-1}f\bigr)(p,q) = \lk(p,q)f(p,q),
  \qquad p,q\in[-\pi,\pi],\; f\in L^2\bigl([-\pi,\pi]^2\bigr),
\end{equation}
where
\[
  \lk(p,q) = \begin{cases}
    4k - i\dfrac{\sin p\cdot\bigl(e^{ikp}-e^{-ikp}\bigr)
    - \sin q\cdot\bigl(e^{ikq}-e^{-ikq}\bigr)}{\cos p-\cos q}\,, & |p|\ne|q|, \\[3ex]
    4k + i\cot p\cdot \bigl(e^{ikp}-e^{-ikp}\bigr) - k\bigl(e^{ikp}+e^{-ikp}\bigr),
    & |p|=|q|\ne0,\pi, \\[2ex]
    0, & p=q=0, \\[2ex]
    4k\bigl(1-(-1)^k\bigr), & |p|=|q|=\pi.
  \end{cases}
\]
Moreover, $\lk$ is continuous and even in both $p$ and $q$,
and the following inequalities hold:
\begin{alignat}{2}
  & 0 \le \lk(p,q) \le 8k, \qquad & & p,q\in[-\pi,\pi],
    \label{boundslk}\\[1ex]
  & \lkarb{1}(p,q) > 0, \qquad & & (p,q)\in[-\pi,\pi]^2\setminus\{(0,0)\}.
    \label{l1pos}
\end{alignat}
\end{lemma}

\begin{proof}
It follows from \eqref{Lkshift} and \eqref{FshiftF} that \eqref{FLkF} holds with
\begin{align}
  \lk(p,q) &= 4k - \sum_{j=0}^{k-1} \Bigl[e^{-i[(k-j)p+jq]} - e^{-i[(-k+j)p-jq]}
    \notag\\[1ex]
  &\hspace*{12ex} - e^{-i[-jp+(k-j)q]} - e^{-i[jp+(-k+j)q]}\Bigr]
    \notag\\[1ex]
  &= 4k - e^{-ikp}\sum_{j=0}^{k-1} e^{ij(p-q)} + e^{ikp}\sum_{j=0}^{k-1} e^{-ij(p-q)}
    \notag\\[1ex]
  &\quad + e^{-ikq}\sum_{j=0}^{k-1} e^{ij(p+q)} + e^{ikq}\sum_{j=0}^{k-1} e^{-ij(p+q)}.
    \label{lksums}
\end{align}
When $|p|\ne|q|$ we can rewrite this as follows:
\begin{align*}
  & \lk(p,q) = 4k - \frac{e^{-ikp}-e^{-ikq}}{1-e^{i(p-q)}}
  - \frac{e^{ikp}-e^{ikq}}{1-e^{-i(p-q)}}
  - \frac{e^{-ikq}-e^{ikp}}{1-e^{i(p+q)}}
  - \frac{e^{ikq}-e^{-ikp}}{1-e^{-i(p+q)}}
    \\[1ex]
  &= 4k - e^{ikp}\biggl(\frac{1}{1-e^{-ip+iq}} - \frac{1}{1-e^{ip+iq}}\biggr)
  - e^{-ikp}\biggl(\frac{1}{1-e^{ip-iq}} - \frac{1}{1-e^{-ip-iq}}\biggr)
    \\[1ex]
  &\quad + e^{ikq}\biggl(\frac{1}{1-e^{-ip+iq}} - \frac{1}{1-e^{-ip-iq}}\biggr)
  + e^{-ikq}\biggl(\frac{1}{1-e^{ip-iq}} - \frac{1}{1-e^{ip+iq}}\biggr).
\end{align*}
The expressions within the brackets can be simplified, e.g.\
\begin{align*}
  & \frac{1}{1-e^{-ip+iq}} - \frac{1}{1-e^{ip+iq}}
  = \frac{e^{-ip+iq}-e^{ip+iq}}{1-e^{-ip+iq}-e^{ip+iq}+e^{2iq}}
    \\[1ex]
  &= \frac{e^{-ip}-e^{ip}}{e^{-iq}-e^{-ip}-e^{ip}+e^{iq}}
  = \frac{i\sin p}{\cos p-\cos q}\,.
\end{align*}
Hence
\begin{align*}
  \lk(p,q) &= 4k - e^{ikp}\frac{i\sin p}{\cos p-\cos q}
  + e^{-ikp}\frac{i\sin p}{\cos p-\cos q}
    \\[1ex]
  &\quad + e^{ikq}\frac{i\sin q}{\cos p-\cos q} - e^{-ikq}\frac{i\sin q}{\cos p-\cos q}
    \\[1ex]
  &= 4k - \frac{i}{\cos p-\cos q}\Bigl[\sin p\cdot\bigl(e^{ikp}-e^{-ikp}\bigr)
  - \sin q\cdot\bigl(e^{ikq}-e^{-ikq}\bigr)\Bigr].
\end{align*}
For the case when $|p|=|q|$ note that $\lk$ is continuous by \eqref{lksums}.
Write $\lk$ as
\[
  \lk(p,q) = 4k-i\frac{f(p)-f(q)}{g(p)-g(q)}
\]
with $f(p)=\sin p\cdot(e^{ikp}-e^{-ikp})$ and $g(p)=\cos p$.
The Generalized Mean Value Theorem implies that
\[
  \lk(p,q) = 4k-i\frac{f'(\xi)}{g'(\xi)}
\]
with $\xi$ between $p$ and $q$.  Hence
\begin{align}
  \lk(p,p) &= \lim_{q\to p}\lk(p,q) = 4k - i\frac{f'(p)}{g'(p)}
    \notag\\[1ex]
  &= 4k - i\frac{\cos p\cdot(e^{ikp}-e^{-ikp})+ik\sin p\cdot(e^{ikp}+e^{-ikp})}{-\sin p}
    \notag\\[1ex]
  &= 4k + i\cot p\cdot\bigl(e^{ikp}-e^{-ikp}\bigr) - k\bigl(e^{ikp}+e^{-ikp}\bigr).
    \label{lkpp}
\end{align}
The relation $\lk(0,0)=0$ follows from \eqref{lksums},
and the value for $\lk(p,q)$ when $|p|=|q|=\pi$ follows from \eqref{lkpp}
by taking the limit $p\to\pi$.

That $\lk$ is even in $p$ and $q$ is clear.
Since $L_k$ is a non-negative operator in $\ell^2(V)$ by \cite[Section~2]{PartI},
the function $\lk$ is non-negative.
The upper bound for $\lk$ in \eqref{boundslk} follows from \eqref{lksums}.

Finally, to show \eqref{l1pos} rewrite $\lkarb{1}$; for $|p|\ne|q|$ we have
\begin{equation}\label{l1simpl}
  \lkarb{1}(p,q) = 4 + 2\frac{\sin^2p-\sin^2q}{\cos p-\cos q}
  = 4 - 2(\cos p + \cos q),
\end{equation}
which extends to all $p,q\in[-\pi,\pi]$ by continuity.
The right-hand side of \eqref{l1simpl} is strictly positive
unless $p=q=0$.
\end{proof}

Let us now consider the Mellin transformation of the $k$-path Laplacians $L_k$, i.e.\
the operator
\[
  \LMs{s} = \sum_{k=1}^\infty \frac{1}{k^s}L_k;
\]
see \eqref{deftLM}.
Since $\|L_k\|\le 8k$ by Lemma~\ref{le:lk}, the series converges in the
operator norm when $s>2$.
As the next lemma shows, the operator $\LMs{s}$ is also unitarily equivalent
to a multiplication operator in $L^2([-\pi,\pi]^2)$.
In order to formulate this lemma, we have to recall the definition
of the polylogarithm.  For $s\in\CC$ the function $\Li_s$ is defined by
\[
  \Li_{s}(z) \defeq \sum_{k=1}^\infty \frac{z^k}{k^s}\,,
  \qquad  |z|<1,
\]
and by analytic continuation to $\CC\setminus[1,\infty)$ with $1$ being a branch point;
see, e.g.\ \cite[25.12.10]{nist}.

\begin{lemma}\label{le:lMs}
For $s>2$ we have
\begin{equation}\label{FLMsF}
  \bigl(\cF\LMs{s}\cF^{-1}f\bigr)(p,q) = \lMs{s}(p,q)f(p,q),
  \quad p,q\in[-\pi,\pi],\; f\in L^2\bigl([-\pi,\pi]^2\bigr),
\end{equation}
where
\begin{align}
  & \lMs{s}(p,q) \defeq \sum_{k=1}^\infty \frac{1}{k^s}\lk(p,q)
    \label{serieslMs}\\[1ex]
  &= \begin{cases}
    4\zeta(s-1) + \dfrac{g_s(p)-g_s(q)}{\cos p-\cos q}\,, & |p|\ne|q|, \\[3ex]
    4\zeta(s-1) - 2\cot p\cdot\Im\bigl(\Li_s(e^{ip})\bigr)
    - 2\Re\bigl(\Li_{s-1}(e^{ip})\bigr), & |p|=|q|\ne0,\pi, \\[1ex]
    0, & p=q=0, \\[1ex]
    4\bigl(1-(-1)^k\bigr)\zeta(s-1), & |p|=|q|=\pi,
  \end{cases}
  \notag
\end{align}
with
\begin{equation}\label{defgs}
  g_s(p) \defeq 2\sin p\cdot\Im\bigl(\Li_s(e^{ip})\bigr).
\end{equation}
The function $\lMs{s}$ is continuous and even in both $p$ and $q$,
and the following inequalities hold:
\begin{alignat}{2}
  & 0 \le \lMs{s}(p,q) \le 8\zeta(s-1), \qquad & & p,q\in[-\pi,\pi],
    \label{boundslMs}\\[1ex]
  & \lMs{s}(p,q) > 0, \qquad & & (p,q)\in[-\pi,\pi]^2\setminus\{(0,0)\}.
    \label{lMspos}
\end{alignat}
\end{lemma}

\begin{proof}
It follows from Lemma~\ref{le:lk} that \eqref{FLMsF} holds with $\lMs{s}$
defined as in \eqref{serieslMs}.
When $|p|\ne|q|$, we have
\begin{align*}
  \lMs{s}(p,q) &= \sum_{k=1}^\infty \frac{1}{k^s}\biggl[4k -
  i\dfrac{\sin p\cdot\bigl(e^{ikp}-e^{-ikp}\bigr) - \sin q\cdot\bigl(e^{ikq}-e^{-ikq}\bigr)}{\cos p-\cos q}\biggr]
    \\[1ex]
  &= 4\sum_{k=1}^\infty \frac{1}{k^{s-1}} - \frac{i}{\cos p-\cos q}
  \biggl[\sin p\cdot\sum_{k=1}^\infty \frac{1}{k^s}\Bigl((e^{ip})^k-(e^{-ip})^k\Bigr)
    \\[1ex]
  &\quad - \sin q\cdot\sum_{k=1}^\infty \frac{1}{k^s}\Bigl((e^{iq})^k-(e^{-iq})^k\Bigr)\biggr]
    \displaybreak[0]\\[1ex]
  &= 4\zeta(s-1) - \frac{i}{\cos p-\cos q}\Bigl[\sin p\cdot\Bigl(\Li_s(e^{ip})-\Li_s(e^{-ip})\Bigr)
    \\[1ex]
  &\quad - \sin q\cdot\Bigl(\Li_s(e^{iq})-\Li_s(e^{-iq})\Bigr)\Bigr],
\end{align*}
which proves the formula for $\lMs{s}$ in the first case.
Now assume that $|p|=|q|\ne0,\pi$.  Then
\begin{align*}
  & \lMs{s}(p,q) = \sum_{k=1}^\infty \frac{1}{k^s}
  \biggl[4k + i\cot p\cdot \bigl(e^{ikp}-e^{-ikp}\bigr) - k\bigl(e^{ikp}+e^{-ikp}\bigr)\biggr]
    \\[1ex]
  &= 4\sum_{k=1}^\infty \frac{1}{k^{s-1}}
  + i\cot p\cdot \sum_{k=1}^\infty \frac{1}{k^s}\Bigl((e^{ip})^k-(e^{-ip})^k\Bigr)
  - \sum_{k=1}^\infty \frac{1}{k^{s-1}}\Bigl((e^{ip})^k+(e^{-ip})^k\Bigr)
    \\[1ex]
  &= 4\zeta(s-1) + i\cot p\cdot \bigl(\Li_s(e^{ip})-\Li_s(e^{-ip})\bigr)
  - \Li_{s-1}(e^{ip}) - \Li_{s-1}(e^{-ip}).
\end{align*}
The remaining cases are clear.

The continuity of $\lMs{s}$ follows from the continuity of $\lk$ and
the fact that the series in \eqref{serieslMs} converges uniformly.
The symmetry of $\lMs{s}$ and the inequalities in \eqref{boundslMs}
follows directly from the symmetry of $\lk$ and \eqref{boundslk}.
The inequality in \eqref{lMspos} follows from \eqref{l1pos}
and the first inequality in \eqref{boundslk}.
\end{proof}

Since $\LMs{s}$ is a bounded operator, the Cauchy problem
\begin{align}
  u'(t) &= -\LMs{s}u(t), \qquad t>0,
    \label{cauchyLMs1}\\[0.5ex]
  u(0) &= \mr{u}
    \label{cauchyLMs2}
\end{align}
has a unique solution, which is given by
\[
  u(t) = e^{-t\LMs{s}}\mr{u}, \qquad t\ge0.
\]
It follows from Lemma~\ref{le:lMs} that
\begin{equation}
\begin{aligned}
  \bigl(\cF e^{-t\LMs{s}}\cF^{-1}f\bigr)(p,q) = e^{-t\lMs{s}(p,q)}f(p,q),
  \hspace*{25ex} & \\[0.5ex]
  \quad t\ge0,\; p,q\in[-\pi,\pi],\; f\in L^2\bigl([-\pi,\pi]^2\bigr). &
\end{aligned}
\end{equation}
Using this relation and the fact that $\lMs{s}$ is even one can easily show
the following theorem; cf.\ \cite[Theorem~5.2]{PartI} for the case of
the infinite path graph.
For the formulation of the theorem let $\bfe_{m,n}\in\ell^2(V)$ be the vector defined by
\begin{equation}\label{defemn}
  (\bfe_{m,n})_{x,y} = \begin{cases}
    1, & m=x,\,n=y, \\[0.5ex]
    0, & \text{otherwise}.
  \end{cases}
\end{equation}

\begin{theorem}\label{th:sol_evol_equ}
Let $s>2$ and $\mr{u}\in\ell^2(V)$.
The unique solution of \eqref{cauchyLMs1}, \eqref{cauchyLMs2} is given by
\[
  u_{x,y}(t) = \frac{1}{4\pi^2} \sum_{m,n\in\ZZ} \mr{u}_{m,n} \int_{-\pi}^\pi \int_{-\pi}^\pi
  e^{i[(x-m)p+(y-n)q]}e^{-t\lMs{s}(p,q)}\rd p\:\rd q,
  \qquad x,y\in\ZZ.
\]
In particular, for $\mr{u}=\bfe_{0,0}$ we obtain
\begin{equation}\label{sole0}
  u_{x,y}(t) = \frac{1}{4\pi^2} \int_{-\pi}^\pi \int_{-\pi}^\pi
  e^{i(xp+yq)}e^{-t\lMs{s}(p,q)}\rd p\:\rd q,
  \qquad x,y\in\ZZ.
\end{equation}
\end{theorem}

\section{Diffusion and superdiffusion for Mellin-transformed $k$-path Laplacian
on a square lattice}
\label{sec:superdiffusion}

\noindent
In this section we examine the long-time behaviour of the solution to
the Cauchy problem generated by the Mellin-transformed $k$-path Laplacian.
The main result is contained in Theorem~\ref{th:super_diffusion}.
To prove this theorem we first examine the asymptotic behaviour of the
function $\lMs{s}$ (see Figure~\ref{lMs_fig}) as the arguments tend to zero, which is
contained in Proposition~\ref{prop:asymp_lMs}.
The discussion is based on similar considerations undertaken for the
path graph in \cite{PartI}, but the arguments are more subtle.
We start with a simple lemma, which is used a couple of times below.

\begin{figure}[ht]
\centering
\begin{tabular}{cc}
\hspace*{-20mm}\includegraphics[width=100mm]{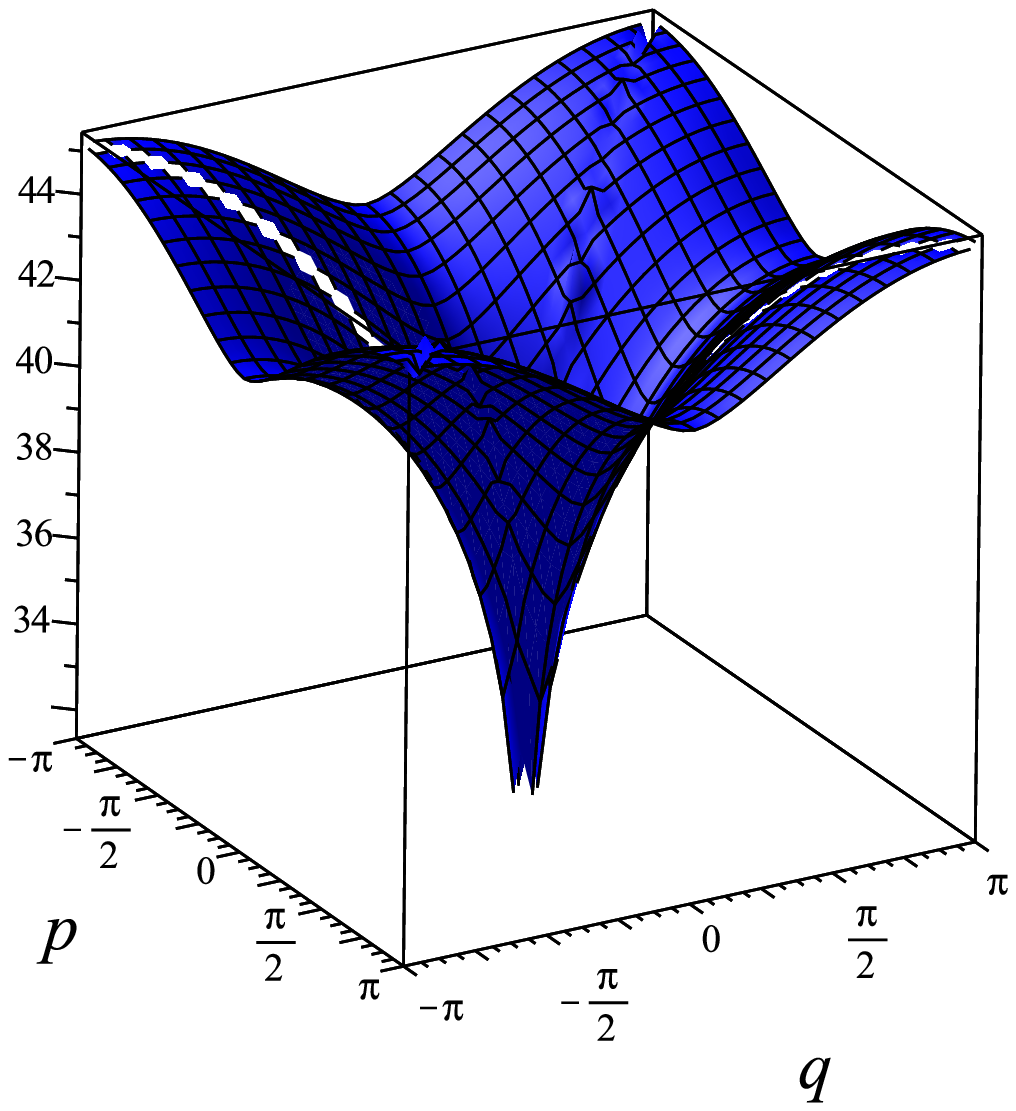}\hspace*{-20mm}
&
\hspace*{-10mm}\includegraphics[width=80mm]{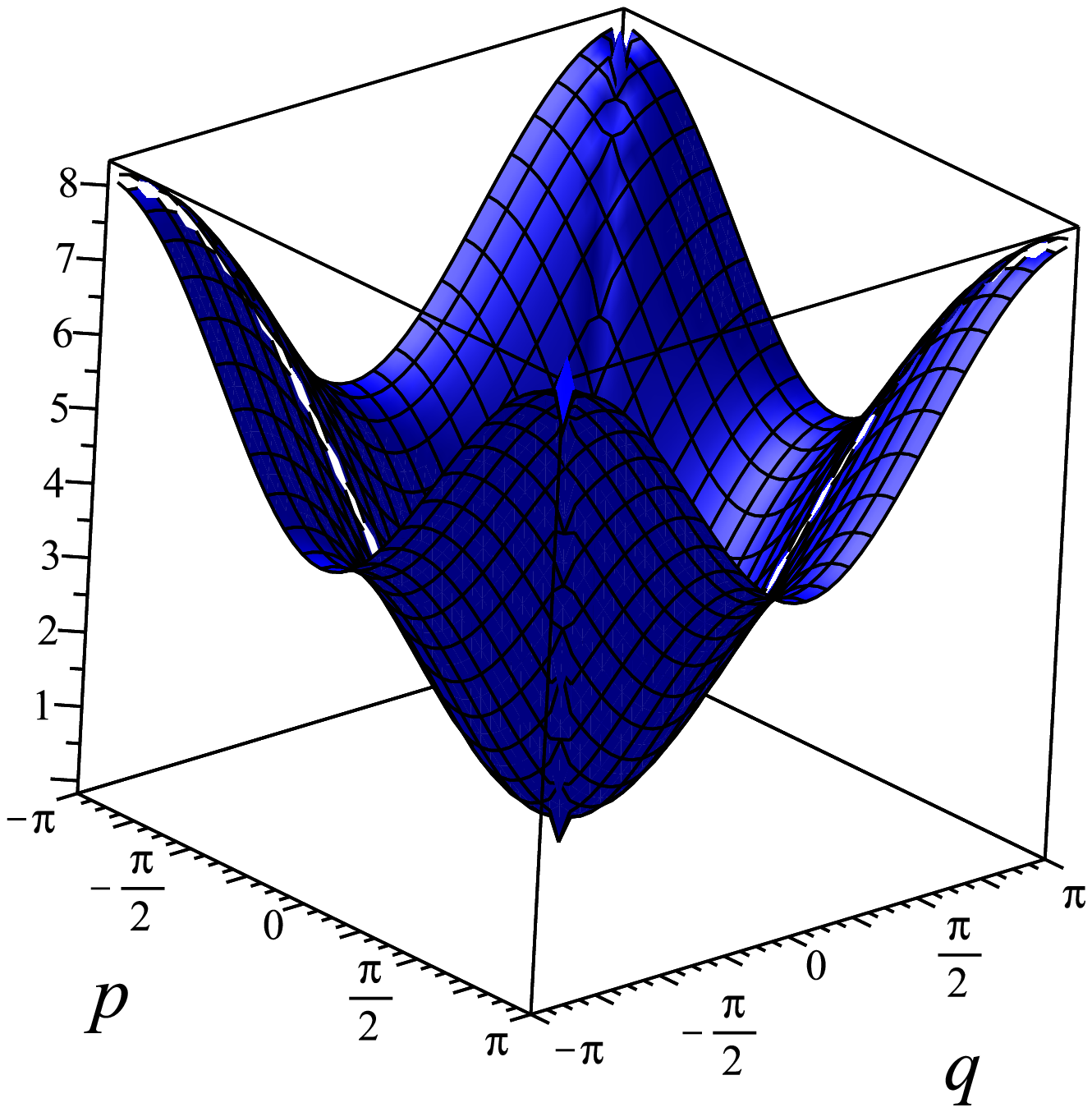}\hspace*{-10mm}
\\
(a) $s=2.1$ & (b) $s=5$
\end{tabular}
\\[1ex]
\includegraphics[width=80mm]{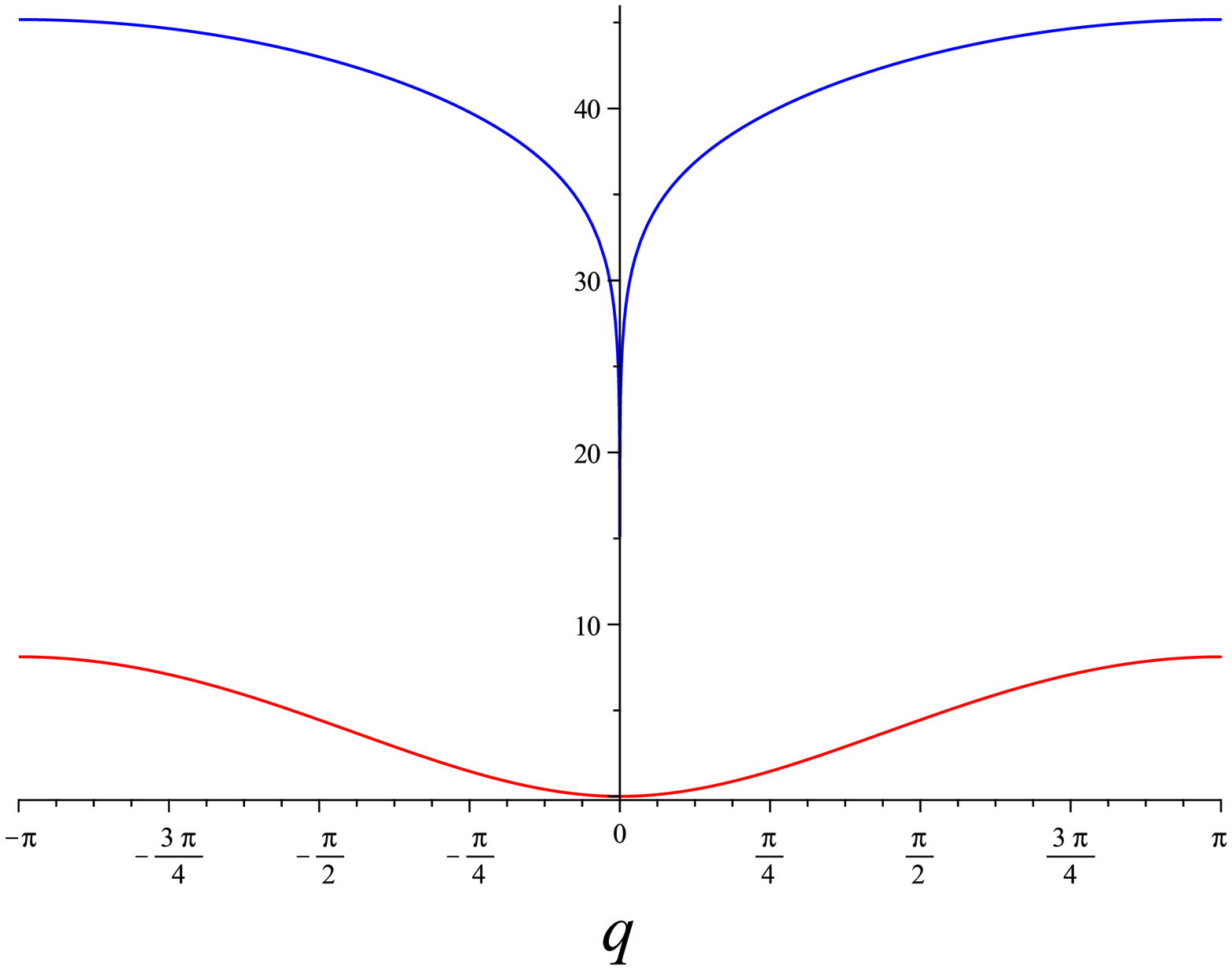}
\\
(c) $\lMs{s}$ restricted to $p=q$ for $s=2.1$ (blue) and $s=5$ (red), respectively
\caption{The graph of the function $\lMs{s}$ on the square $[-\pi,\pi]^{2}$
for the parameters $s=2.1$ (a) and $s=5$ (b), respectively.
The third graph (c) shows the behaviour of both solutions $s=2.1$ and $s=5$
restricted to the line $p=q$.}
\label{lMs_fig}
\end{figure}

\begin{lemma}\label{le:diff_quot}
Let $b>0$, let $f:(0,b)\to\RR$ be differentiable and assume that
\[
  |f'(t)| \le Ct^\alpha, \qquad t\in(0,b),
\]
for some $C>0$ and $\alpha\ge1$.  Then
\[
  \bigg|\frac{f(p)-f(q)}{p^2-q^2}\bigg|
  \le \frac{C}{2}\max\bigl\{p^{\alpha-1},q^{\alpha-1}\bigr\},
  \qquad p,q\in(0,b),\, p\ne q.
\]
\end{lemma}

\begin{proof}
Define the function $g(x)\defeq f(\sqrt{x})$, $x\in(0,b^2)$.
Let $p,q\in(0,b)$ such that $p\ne q$ and set $x\defeq p^2$, $y\defeq q^2$.  Then
\begin{equation}\label{diff_quotient}
  \bigg|\frac{f(p)-f(q)}{p^2-q^2}\bigg|
  = \bigg|\frac{g(x)-g(y)}{x-y}\bigg| \\[1ex]
  = |g'(\xi)|
\end{equation}
for some $\xi$ between $x$ and $y$ by the Mean Value Theorem.
Since $\sqrt{\xi}\le\max\{p,q\}$, we obtain that
\[
  |g'(\xi)| = \bigg|\frac{f'(\sqrt{\xi})}{2\sqrt{\xi}}\bigg|
  \le \frac{C(\sqrt{\xi})^\alpha}{2\sqrt{\xi}}
  \le \frac{C}{2}\max\bigl\{p^{\alpha-1},q^{\alpha-1}\bigr\},
\]
\vspace*{-2ex}
which, together with \eqref{diff_quotient}, finishes the proof.
\end{proof}

In the next three lemmas we prove auxiliary asymptotic results,
which are used to obtain the asymptotic behaviour of $\lMs{s}$
in Proposition~\ref{prop:asymp_lMs}.

\begin{lemma}\label{le:asymp1}
We have
\[
  \frac{1}{\cos p-\cos q} = -\frac{2}{p^2-q^2}\biggl[1+\frac{1}{12}\bigl(p^2+q^2\bigr)
  + R_1(p,q)\biggr], \quad
  p,q\in[-\pi,\pi],\; |p|\ne|q|.
\]
where
\[
  R_1(p,q) = \rmO\bigl(p^4+q^4\bigr), \qquad p,q\to0,\; |p|\ne|q|.
\]
\end{lemma}

\begin{proof}
We write
\[
  \cos p = 1 - \frac{p^2}{2} + \frac{p^4}{24} + f(p)
\]
where $f'(p) = \rmO(p^5)$, $p\to0$.  For $p,q\in(0,\pi]$ with $p\ne q$ we have
\begin{align*}
  \cos p - \cos q &= -\frac{p^2-q^2}{2} + \frac{p^4-q^4}{24} + f(p) - f(q) \\[1ex]
  &= -\frac{1}{2}(p^2-q^2)\biggl[1-\frac{1}{12}\bigl(p^2+q^2\bigr)
  - 2\frac{f(p)-f(q)}{p^2-q^2}\biggr] \displaybreak[0]\\[1ex]
  &= -\frac{1}{2}(p^2-q^2)\biggl[1-\frac{1}{12}\bigl(p^2+q^2\bigr)
  + \rmO\bigl(p^4+q^4\bigr)\biggr], \qquad p,q\to0,
\end{align*}
where the last relation follows from Lemma~\ref{le:diff_quot}.
Now the claim is obtained by taking inverses on both sides
and extending the result to non-positive $p,q$ by continuity and symmetry.
\end{proof}

\begin{lemma}\label{le:asymp2}
Let $s\in (2,\infty)\setminus\{4\}$.  Then
\[
  2\Im\bigl(\Li_s(e^{ip})\bigr) = -\frac{C_s}{2}p^{s-1} + 2\zeta(s-1)p
  - \frac{\zeta(s-3)}{3}p^3 + R_{2,s}(p), \qquad p \in (0,2\pi),
\]
where
\begin{equation}\label{defCs}
  C_s \defeq \begin{cases}
    -\dfrac{2\pi}{\Gamma(s)\sin(\frac{s\pi}{2})}\,, & s\notin2\ZZ, \\[2ex]
    0, & s\in2\ZZ,
  \end{cases}
\end{equation}
and
\[
  R_{2,s}(p) = \rmO(p^5) \quad\text{and}\quad R'_{2,s}(p) = \rmO(p^4),
  \qquad p\searrow0, \qquad \text{if}\;\; s\ne6,
\]
and
\[
  R_{2,s}(p) = \rmO\bigl(p^5|\ln p|\bigr) \quad\text{and}\quad
  R'_{2,s}(p) = \rmO\bigl(p^4|\ln p|\bigr),
  \qquad p\searrow0, \qquad \text{if}\;\; s=6.
\]
\end{lemma}

\begin{proof}
First let $s\in(2,\infty)\setminus\NN$.  It follows from \cite[25.12.12]{nist} that,
for $p\in(0,2\pi)$,
\begin{align*}
  & 2\Im\bigl(\Li_s(e^{ip})\bigr)
  = 2\Im\Biggl[\Gamma(1-s)(-ip)^{s-1} + \sum_{n=0}^\infty \zeta(s-n)\frac{(ip)^n}{n!}\Biggr] \\[1ex]
  &= 2\Im\Biggl[\Gamma(1-s)p^{s-1}e^{-(s-1)\frac{\pi}{2}i}
  + \sum_{n=0}^\infty \zeta(s-n)\frac{i^n p^n}{n!}\Biggr] \\[1ex]
  &= -2\Gamma(1-s)\sin\Bigl((s-1)\frac{\pi}{2}\Bigr)p^{s-1}
  + 2\sum_{l=0}^\infty \zeta(s-2l-1)\frac{(-1)^l}{(2l+1)!}p^{2l+1} \\[1ex]
  &= -\frac{C_s}{2}p^{s-1} + 2\zeta(s-1)p - \frac{\zeta(s-3)}{3}p^3 + R_{2,s}(p),
\end{align*}
where
\begin{align*}
  C_s &= 4\Gamma(1-s)\sin\Bigl((s-1)\frac{\pi}{2}\Bigr)
  = \frac{4\pi\sin\bigl((s-1)\frac{\pi}{2}\bigr)}{\Gamma(s)\sin(s\pi)}
  = -\frac{4\pi\cos(\frac{s\pi}{2})}{\Gamma(s)\sin(s\pi)} \\[1ex]
  &= -\frac{2\pi}{\Gamma(s)\sin(\frac{s\pi}{2})}
\end{align*}
and
\[
  R_{2,s}(p) = 2\sum_{l=2}^\infty \zeta(s-2l-1)\frac{(-1)^l}{(2l+1)!}p^{2l+1}.
\]
This relation extends to $s$ being an odd integer with $s\ge3$.
Moreover, $R_{2,s}$ satisfies
\[
  R_{2,s}(p) = \rmO(p^5) \quad\text{and}\quad R'_{2,s}(p) = \rmO(p^4), \qquad p\searrow0.
\]
This proves the claim for $s\in(2,\infty)\setminus2\NN$.

Now let $s\in\{6,8,\ldots\}$ and set
\[
  H_n = \sum_{j=1}^n \frac{1}{j}\,.
\]
From \cite[p.~131]{lindeloef} we obtain, again for $p\in(0,2\pi)$,
\begin{align*}
  & 2\Im\bigl(\Li_s(e^{ip})\bigr)
  = 2\Im\Biggl[\frac{(ip)^{s-1}}{(s-1)!}\Bigl(H_{s-1}-\Log(-ip)\Bigr)
  + \sum_{\substack{n=0 \\[0.2ex] n\ne s-1}}^\infty \zeta(s-n)\frac{(ip)^n}{n!}\Biggr] \\[1ex]
  &= 2\Im\Biggl[\frac{(-1)^{\frac{s}{2}-1}i p^{s-1}}{(s-1)!}\Bigl(H_{s-1} - \ln p + i\frac{\pi}{2}\Bigr)
  + \sum_{\substack{n=0 \\[0.2ex] n\ne s-1}}^\infty \zeta(s-n)\frac{i^n p^n}{n!}\Biggr] \\[1ex]
  &= \frac{2(-1)^{\frac{s}{2}-1}}{(s-1)!}p^{s-1}\bigl(H_{s-1}-\ln p\bigr)
  + 2\sum_{\substack{l=0 \\[0.2ex] l\ne\frac{s}{2}-1}}^\infty
  \zeta(s-2l-1)\frac{(-1)^l}{(2l+1)!}p^{2l+1} \\[1ex]
  &= 2\zeta(s-1)p - \frac{\zeta(s-3)}{3}p^3 + R_{2,s}(p),
\end{align*}
where
\[
  R_{2,s}(p) = \frac{2(-1)^{\frac{s}{2}-1}}{(s-1)!}p^{s-1}\bigl(H_{s-1}-\ln p\bigr)
  + 2\sum_{\substack{l=2 \\[0.2ex] l\ne\frac{s}{2}-1}}
  \zeta(s-2l-1)\frac{(-1)^l}{(2l+1)!}p^{2l+1},
\]
which satisfies
\[
  R_{2,s}(p) = \rmO(p^5) \quad\text{and}\quad R'_{2,s}(p) = \rmO(p^4),
  \qquad p\searrow0, \qquad \text{if}\;\; s\ge8,
\]
and
\[
  R_{2,s}(p) = \rmO\bigl(p^5|\ln p|\bigr) \quad\text{and}\quad
  R'_{2,s}(p) = \rmO\bigl(p^4|\ln p|\bigr),
  \qquad p\searrow0, \qquad \text{if}\;\; s=6.
\]
This finishes the proof in the case when $s\in\{6,8,\ldots\}$.
\end{proof}

\begin{lemma}\label{le:asymp3}
Let $s\in(2,\infty)\setminus\{4\}$ and let $g_s$ be defined as in \eqref{defgs}
and $C_s$ as in \eqref{defCs}.  Then
\[
  g_s(p) = - \frac{C_s}{2}p^s + 2\zeta(s-1)p^2 - \frac{\zeta(s-1)+\zeta(s-3)}{3}p^4 + R_{3,s}(p),
\]
where $R_{3,s}$ satisfies
\begin{equation}\label{asympR2}
  R'_{3,s}(p) = \begin{cases}
    \rmO(p^{s+1}), & s\in(2,4), \\[1ex]
    \rmO(p^5), & s\in(4,\infty)\setminus\{6\}, \\[1ex]
    \rmO\bigl(p^5|\ln p|\bigr), & s=6,
  \end{cases}
\end{equation}
as $p\searrow0$.
\end{lemma}

\begin{proof}
Write $\sin p = p-\frac{p^3}{6}+\Rsin(p)$.
From Lemma~\ref{le:asymp2} we obtain that
\begin{align*}
  g_s(p) &= 2\sin p\cdot\Im\bigl(\Li_s(e^{ip})\bigr) \\[1ex]
  &= \biggl[p-\frac{p^3}{6}+\Rsin(p)\biggr]
  \biggl[-\frac{C_s}{2}p^{s-1} + 2\zeta(s-1)p - \frac{\zeta(s-3)}{3}p^3 + R_{2,s}(p)\biggr] \\[1ex]
  &= - \frac{C_s}{2}p^s + 2\zeta(s-1)p^2 - \frac{\zeta(s-1)+\zeta(s-3)}{3}p^4 + R_{3,s}(p)
\end{align*}
where
\begin{align*}
  R_{3,s}(p) &= \frac{C_s}{12}p^{s+2} - \frac{C_s}{2}p^{s-1}\Rsin(p)
  + 2\zeta(s-1)p\Rsin(p) \\[1ex]
  &\quad + \frac{\zeta(s-3)}{18}p^6 - \frac{\zeta(s-3)}{3}p^3\Rsin(p)
  + \sin p\cdot R_{2,s}(p),
\end{align*}
which satisfies
\[
  R_{3,s}'(p) = \rmO(p^{s+1}) + \rmO(p^5) + \rmO\bigl(R_{2,s}(p)\bigr)
  + \rmO\bigl(pR_{2,s}'(p)\bigr).
\]
The latter relation yields \eqref{asympR2}.
\end{proof}

In the next proposition we consider the asymptotic behaviour of the
function $\lMs{s}$ around the origin.
In particular, we observe that the behaviour differs for the two cases
$s\in(2,4)$ and $s\in(4,\infty)$.
For the case when $s=4$ the behaviour is more complicated and involves
a logarithmic term; we do not consider this case in the following.

\begin{proposition}\label{prop:asymp_lMs}
Let $s\in(2,\infty)\setminus\{4\}$, let $\lMs{s}$ be as in \eqref{serieslMs}
and $C_s$ as in \eqref{defCs}.
Moreover, define
\begin{align*}
  h_{1,s}(p,q) &\defeq \begin{cases}
    C_s\dfrac{|p|^s-|q|^s}{p^2-q^2}\,, & |p|\ne|q|, \\[3ex]
    \dfrac{sC_s}{2}|p|^{s-2}, & |p|=|q|,
  \end{cases}
  \\[1ex]
  h_{2,s}(p,q) &\defeq \frac{\zeta(s-1)+2\zeta(s-3)}{3}(p^2+q^2).
\end{align*}
Then
\[
  \lMs{s}(p,q) = h_{1,s}(p,q) + h_{2,s}(p,q) + R_s(p,q),
  \qquad p,q\in[-\pi,\pi],
\]
where
\[
  R_s(p,q) = \rmO\bigl(p^\alpha+q^\alpha\bigr), \qquad p,q\to0,
\]
with
\[
  \alpha = \begin{cases}
    \min\{s,4\}, & s\ne6, \\[1ex]
    4-\eps, & s=6,
  \end{cases}
\]
with an arbitrary $\eps>0$.

In particular, we have
\[
  \lMs{s}(p,q) = \begin{cases}
    h_{1,s}(p,q) + \rmO(p^s+q^s), & s\in(2,4), \\[1ex]
    h_{2,s}(p,q) + \rmO(p^4+q^4), & s\in(4,\infty)\setminus\{6\}, \\[1ex]
    h_{2,s}(p,q) + \rmO(p^{4-\eps}+q^{4-\eps}), & s=6,
  \end{cases}
\]
as $p,q\to0$ with arbitrary $\eps>0$ when $s=6$.
\end{proposition}

\begin{proof}
Let $p,q\in(0,\pi]$ such that $p\ne q$.
From Lemmas~\ref{le:lMs}, \ref{le:asymp1} and \ref{le:asymp3} we obtain
\begin{align*}
  & \lMs{s}(p,q) = 4\zeta(s-1) + \frac{g_s(p)-g_s(q)}{\cos p-\cos q} \\[1ex]
  &= 4\zeta(s-1) - \frac{2}{p^2-q^2}\biggl[1+\frac{1}{12}\bigl(p^2+q^2\bigr)
  +R_1(p,q)\biggr]\times
    \\[1ex]
  &\quad \times\biggl[-\frac{C_s}{2}(p^s-q^s) + 2\zeta(s-1)(p^2-q^2)
  - \frac{\zeta(s-1)+\zeta(s-3)}{3}\bigl(p^4-q^4\bigr)
    \\[1ex]
  &\qquad+ R_{3,s}(p) - R_{3,s}(q)\biggr]
    \displaybreak[0]\\[1ex]
  &= 4\zeta(s-1) + \biggl[1+\frac{1}{12}\bigl(p^2+q^2\bigr)+R_1(p,q)\biggr]
  \cdot\biggl[C_s\frac{p^s-q^s}{p^2-q^2}
    \\[1ex]
  &\quad - 4\zeta(s-1) + \frac{2}{3}\Bigl(\zeta(s-1)+\zeta(s-3)\Bigr)(p^2+q^2)
  - 2\frac{R_{3,s}(p)-R_{3,s}(q)}{p^2-q^2}\biggr]
    \\[1ex]
  &= C_s\frac{p^s-q^s}{p^2-q^2} + \biggl[\frac{2}{3}\Bigl(\zeta(s-1)+\zeta(s-3)\Bigr)
  -\frac{1}{3}\zeta(s-1)\biggr](p^2+q^2) + R_s(p,q),
\end{align*}
where
\begin{align*}
  R_s(p,q) &= C_s\biggl[\frac{1}{12}\bigl(p^2+q^2\bigr)+R_1(p,q)\biggr]
  \frac{p^s-q^s}{p^2-q^2}
    \\[1ex]
  &\quad + R_1(p,q)\biggl[-4\zeta(s-1)+\frac{2}{3}\Bigl(\zeta(s-1)+\zeta(s-3)\Bigr)
  (p^2+q^2)\biggr]
    \\[1ex]
  &\quad + \frac{1}{18}\Bigl(\zeta(s-1)+\zeta(s-3)\Bigr)(p^2+q^2)^2
    \\[1ex]
  &\quad - 2\biggl[1+\frac{1}{12}\bigl(p^2+q^2\bigr)+R_1(p,q)\biggr]
  \frac{R_{3,s}(p)-R_{3,s}(q)}{p^2-q^2}\,.
\end{align*}
It follows from Lemma~\ref{le:asymp3} that
\begin{equation}\label{asympR3s}
  R'_{3,s}(p) = \rmO(p^\beta) \qquad\text{where}\;\;
  \beta = \begin{cases}
    \min\{s+1,5\}, & s\ne6, \\[1ex]
    5-\eps, & s=6,
  \end{cases}
\end{equation}
for arbitrary $\eps>0$.
Lemma~\ref{le:diff_quot} implies that
\[
  \frac{p^s-q^s}{p^2-q^2} = \rmO\bigl(p^{s-2}+q^{s-2}\bigr), \qquad q,p\to0,\; p\ne q,
\]
and
\[
  \frac{R_{3,s}(p)-R_{3,s}(q)}{p^2-q^2} = \rmO(p^{\beta-1}), \qquad q,p\to0,\; p\ne q,
\]
where $\beta$ is as in \eqref{asympR3s}.
The error term $R_s$ satisfies
\[
  R_s(p,q) = \rmO\bigl(p^\alpha+q^\alpha\bigr), \qquad p,q\to0,\; p\ne q,
\]
where
\[
  \alpha = \begin{cases}
    \min\{s,4\}, & s\ne6, \\[1ex]
    4-\eps, & s=6,
  \end{cases}
\]
with an arbitrary $\eps>0$.
Since $\lMs{s}$, $h_{1,s}$ and $h_{2,s}$ are continuous and even in $p$ and $q$,
the result extends to all $p,q\in[-\pi,\pi]$.
\end{proof}

The next lemma is the key lemma about the long-time behaviour
of the solution of the Cauchy problem;
it is a generalization of \cite[Lemma~6.1]{PartI} to the two-dimensional setting.
It is more subtle than the one-dimensional case, but a further generalization to
$n$ dimensions is straightforward.

\begin{lemma}\label{le:general_asymp}
Let $\alpha>0$ and let $l:[-\pi,\pi]^2\to\RR$ be a continuous function that satisfies
\begin{equation}\label{lpos}
  \lns(p,q) > 0, \qquad (p,q)\in[-\pi,\pi]^2\setminus\{(0,0)\}
\end{equation}
and can be written as
\[
  \lns(p,q) = h(p,q) + R(p,q)
\]
where the continuous function $h:\RR^2\to\RR$ satisfies
\[
  h(rp,rq) = r^\alpha h(p,q), \qquad r>0,\;p,q\in\RR,
\]
and
\begin{equation}\label{remainder}
  R(p,q) = \rmo\bigl(|p|^\alpha+|q|^\alpha\bigr), \qquad p,q\to0.
\end{equation}
Define the function
\[
  f(x,y,t) \defeq \frac{1}{4\pi^2}\int\limits_{-\pi}^\pi \int\limits_{-\pi}^\pi
  e^{i(xp+yq)}e^{-t\lns(p,q)}\rd p\,\rd q,
  \qquad x,y\in\RR.
\]
Then
\begin{align}
  t^{\frac{2}{\alpha}}f\bigl(t^{\frac{1}{\alpha}}\xi,t^{\frac{1}{\alpha}}\eta,t\bigr)
  \to \frac{1}{4\pi^2} \int\limits_{-\infty}^\infty \int\limits_{-\infty}^\infty
  e^{i(\xi v+\eta w)}e^{-h(v,w)}\rd v\,\rd w \eqdef F(\xi,\eta),
    \hspace*{10ex} & \label{conv_lemma} \\[1ex]
  t\to\infty,\;\text{uniformly in}\;\; \xi,\eta\in\RR. & \notag
\end{align}
Hence
\begin{align}
  f(x,y) = t^{-\frac{2}{\alpha}}F\bigl(t^{-\frac{1}{\alpha}}x,t^{-\frac{1}{\alpha}}y\bigr)
  + \rmo\bigl(t^{-\frac{2}{\alpha}}\bigr),
    \hspace*{20ex} & \label{conv_lemma2} \\[1ex]
  t\to\infty,\; \text{uniformly in}\;\; x,y\in\RR. & \notag
\end{align}
\end{lemma}

\begin{proof}
Let us first show that there exists $C>0$ such that
\begin{equation}\label{llowerbd}
  \lns(p,q) \ge C\bigl(|p|^\alpha+|q|^\alpha\bigr),
  \qquad p,q\in[-\pi,\pi].
\end{equation}
For fixed $(p,q)\in\RR^2\setminus\{(0,0)\}$ we have
\[
  \lns(rp,rq) = r^\alpha h(p,q) + \rmo(r^\alpha), \qquad r\searrow0,
\]
which, together with \eqref{lpos} implies that $h(p,q)>0$
for $(p,q)\in\RR^2\setminus\{(0,0)\}$.
Set
\[
  C_1 \defeq \min_{|p|^\alpha+|q|^\alpha=1}h(p,q),
\]
which is a positive number.
Let $(p,q)\in\RR^2\setminus\{(0,0)\}$ and
set $r\defeq (|p|^\alpha+|q|^\alpha)^{\frac{1}{\alpha}}$.  Then
\[
  h(p,q) = h\Bigl(r\frac{p}{r},r\frac{q}{r}\Bigr)
  = r^\alpha h\Bigl(\frac{p}{r},\frac{q}{r}\Bigr)
  \ge C_1r^\alpha
\]
and hence
\begin{equation}\label{hlowerbd}
  h(p,q) \ge C_1\bigl(|p|^\alpha+|q|^\alpha\bigr),
  \qquad p,q\in\RR.
\end{equation}
Together with \eqref{remainder}, this implies that
\[
  \lns(p,q) \ge \frac{C_1}{2}\bigl(|p|^\alpha+|q|^\alpha\bigr),
  \qquad p,q\in\RR \;\;\text{such that}\;\; |p|^\alpha+|q|^\alpha \le r_0
\]
for some $r_0>0$.
Since $\lns$ is continuous and satisfies \eqref{lpos}, we obtain \eqref{llowerbd}.

For $\xi,\eta\in\RR$ and $t>0$ we can use the substitution $v=t^{\frac{1}{\alpha}}p$,
$w=t^{\frac{1}{\alpha}}q$ to obtain
\begin{align*}
  t^{\frac{2}{\alpha}}f\bigl(t^{\frac{1}{\alpha}}\xi,t^{\frac{1}{\alpha}}\eta,t\bigr)
  &= t^{\frac{2}{\alpha}}\frac{1}{4\pi^2} \int\limits_{-\pi}^\pi \int\limits_{-\pi}^\pi
  e^{it^{\frac{1}{\alpha}}(\xi p+\eta q)}e^{-t\lns(p,q)}\rd p\,\rd q
    \\[1ex]
  &= \frac{1}{4\pi^2}\int\limits_{-t^{\frac{1}{\alpha}}\pi}^{t^{\frac{1}{\alpha}}\pi}\;
  \int\limits_{-t^{\frac{1}{\alpha}}\pi}^{t^{\frac{1}{\alpha}}\pi}
  e^{i(\xi v+\eta w)}e^{-t\lns(t^{-\frac{1}{\alpha}}v,t^{-\frac{1}{\alpha}}w)}\rd v\,\rd w.
\end{align*}
Hence
\begin{align}
  & \Big|t^{\frac{2}{\alpha}}f\bigl(t^{\frac{1}{\alpha}}\xi,t^{\frac{1}{\alpha}}\eta,t\bigr)
  - F(\xi,\eta)\Big|
    \notag\\[1ex]
  &= \Bigg|\frac{1}{4\pi^2}\hspace*{-2ex}
  \iint\limits_{[-t^{\frac{1}{\alpha}}\pi,t^{\frac{1}{\alpha}}\pi]^2}
  \hspace*{-2ex} e^{i(\xi v+\eta w)}
  \Bigl(e^{-t\lns(t^{-\frac{1}{\alpha}}v,t^{-\frac{1}{\alpha}}w)}-e^{-h(v,w)}\Bigr)\rd v\,\rd w
    \notag\\[1ex]
  &\quad - \frac{1}{4\pi^2}\hspace*{-1ex}
  \iint\limits_{\RR^2\setminus[-t^{\frac{1}{\alpha}}\pi,t^{\frac{1}{\alpha}}\pi]^2}
  \hspace*{-1ex} e^{i(\xi v+\eta w)}e^{-h(v,w)}\rd v\,\rd w\Bigg|
    \displaybreak[0]\notag\\[1ex]
  &\le \frac{1}{4\pi^2}\iint\limits_{\RR^2}
  \chi_{[-t^{\frac{1}{\alpha}}\pi,t^{\frac{1}{\alpha}}\pi]^2}(v,w)
  \Big|e^{-t\lns(t^{-\frac{1}{\alpha}}v,t^{-\frac{1}{\alpha}}w)}-e^{-h(v,w)}\Bigr|\,\rd v\,\rd w
    \label{int1}\\[1ex]
  &\quad + \frac{1}{4\pi^2}\hspace*{-1ex}
  \iint\limits_{\RR^2\setminus[-t^{\frac{1}{\alpha}}\pi,t^{\frac{1}{\alpha}}\pi]^2}
  \hspace*{-2ex} e^{-h(v,w)}\rd v\,\rd w,
    \label{int2}
\end{align}
where $\chi_G$ is the characteristic function of a set $G\subseteq\RR^2$.
The integral in \eqref{int2} converges to $0$ as $t\to\infty$;
note that the integral in \eqref{int2} exists by \eqref{hlowerbd}.
From \eqref{llowerbd} and \eqref{hlowerbd} we obtain the
following estimate for the integrand in \eqref{int1}:
\begin{align*}
  & \chi_{[-t^{\frac{1}{\alpha}}\pi,t^{\frac{1}{\alpha}}\pi]^2}(v,w)
  \Big|e^{-t\lns(t^{-\frac{1}{\alpha}}v,t^{-\frac{1}{\alpha}}w)}-e^{-h(v,w)}\Bigr|
    \\[1ex]
  &\le \chi_{[-t^{\frac{1}{\alpha}}\pi,t^{\frac{1}{\alpha}}\pi]^2}(v,w)
  \Bigl(e^{-t\lns(t^{-\frac{1}{\alpha}}v,t^{-\frac{1}{\alpha}}w)}
  + e^{-h(v,w)}\Bigr)
    \\[1ex]
  &\le \chi_{[-t^{\frac{1}{\alpha}}\pi,t^{\frac{1}{\alpha}}\pi]^2}(v,w)
  \Bigl(e^{-tC(t^{-1}|v|^\alpha+t^{-1}|w|^\alpha)}+e^{-C_1(|v|^\alpha+|w|^\alpha)}\Bigr)
    \\[1ex]
  &\le e^{-C(|v|^\alpha+|w|^\alpha)} + e^{-C_1(|v|^\alpha+|w|^\alpha)},
\end{align*}
where the right-hand side is integrable on $\RR^2$ and independent of $t$.
For fixed $v,w\in\RR^2$ and large enough $t>0$ we have
\begin{align*}
  t\lns\bigl(t^{-\frac{1}{\alpha}}v,t^{-\frac{1}{\alpha}}w\bigr)
  &= th\bigl(t^{-\frac{1}{\alpha}}v,t^{-\frac{1}{\alpha}}w\bigr)
  + tR\bigl(t^{-\frac{1}{\alpha}}v,t^{-\frac{1}{\alpha}}w\bigr)
    \\[1ex]
  &= h(v,w) + t\rmo\bigl(t^{-1}\bigl(|v|^\alpha+|w|^\alpha\bigr)
  \to h(v,w) \qquad \text{as}\;\; t\to\infty.
\end{align*}
Hence the integrand in \eqref{int1} converges to $0$ pointwise as $t\to\infty$.
Now the Dominated Convergence Theorem implies that the integral in \eqref{int1}
converges to $0$ as $t\to\infty$.
Since the integrals in \eqref{int1} and \eqref{int2} are independent of $\xi$ and $\eta$,
the convergence in \eqref{conv_lemma} is uniform in $\xi$ and $\eta$.

The relation in \eqref{conv_lemma2} follows easily from \eqref{conv_lemma}
by using the substitution $x=t^{\frac{1}{\alpha}}\xi$, $y=t^{\frac{1}{\alpha}}\eta$.
\end{proof}

The next theorem is the main result of the paper.
It contains the long-time behaviour of the solution of the Cauchy problem
corresponding to the Mellin-transformed $k$-path Laplacian.
It shows, in particular, that, for $s\in(2,4)$, the solution exhibits
superdiffusive behaviour whereas for $s>4$ one has normal diffusion.

\begin{theorem}\label{th:super_diffusion}
Let $\Gamma=(V,E)$ be the square lattice as described at the beginning
of Section~\ref{sec:time_evolution}, let $s>2$, $s\ne4$, and let $\LMs{s}$
be the Mellin-transformed $k$-path Laplacian defined in \eqref{deftLM}.
Let $u$ be the solution in \eqref{sole0} of \eqref{cauchyLMs1}, \eqref{cauchyLMs2}
with $\mr u=\bfe_{0,0}$, where $\bfe_{0,0}$ is defined in \eqref{defemn}.
Then
\[
  u_{x,y}(t) = t^{-\frac{2}{\alpha}}F_s\bigl(t^{-\frac{1}{\alpha}}x,t^{-\frac{1}{\alpha}}y\bigr)
  + \rmo\bigl(t^{-\frac{2}{\alpha}}\bigr),
  \qquad t\to\infty,\; \text{uniformly in}\;\; x,y\in\ZZ,
\]
where in the case $s\in(2,4)$,
\[
  \alpha=s-2 \qquad\text{and}\qquad
  F_s(\xi,\eta) \defeq \frac{1}{4\pi^2}\int\limits_{-\infty}^\infty \int\limits_{-\infty}^\infty
  e^{i(\xi v+\eta w)}e^{-h_{1,s}(v,w)}\rd v\,\rd w
\]
with $h_{1,s}$ from Proposition~\ref{prop:asymp_lMs},
and in the case $s\in(4,\infty)$,
\[
  \alpha = 2 \qquad\text{and}\qquad
  F_s(\xi,\eta) \defeq \frac{1}{4\pi\gamma_s}e^{-\frac{\xi^2+\eta^2}{4\gamma_s}}
\]
with
\[
  \gamma_s = \frac{\zeta(s-1)+2\zeta(s-3)}{3}\,.
\]
(See Figures~\ref{fig:s_dep} and \ref{fig:Fs}.)
\end{theorem}

\begin{proof}
By Proposition~\ref{prop:asymp_lMs} and Lemma~\ref{le:lMs} the function $\lMs{s}$
satisfies the assumptions of Lemma~\ref{le:general_asymp}
with $h=h_{1,s}$ and $\alpha=s-2$ when $s\in(2,4)$
and with $h=h_{2,s}$ and $\alpha=2$ when $s\in(4,\infty)$.
Hence all claims follow from Lemma~\ref{le:general_asymp}.
\end{proof}

\begin{center}
\begin{figure}[ht]
\includegraphics[width=.50\textwidth]{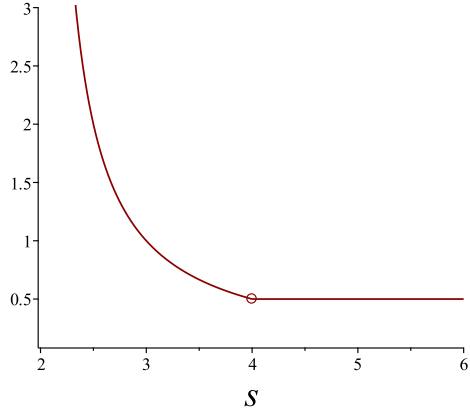}
\caption{The $s$-dependance of $\frac{1}{\alpha}$.}
\label{fig:s_dep}
\end{figure}
\end{center}

\begin{figure}[ht]
\centering
\begin{tabular}{cc}
\hspace*{-5mm}\includegraphics[width=65mm]{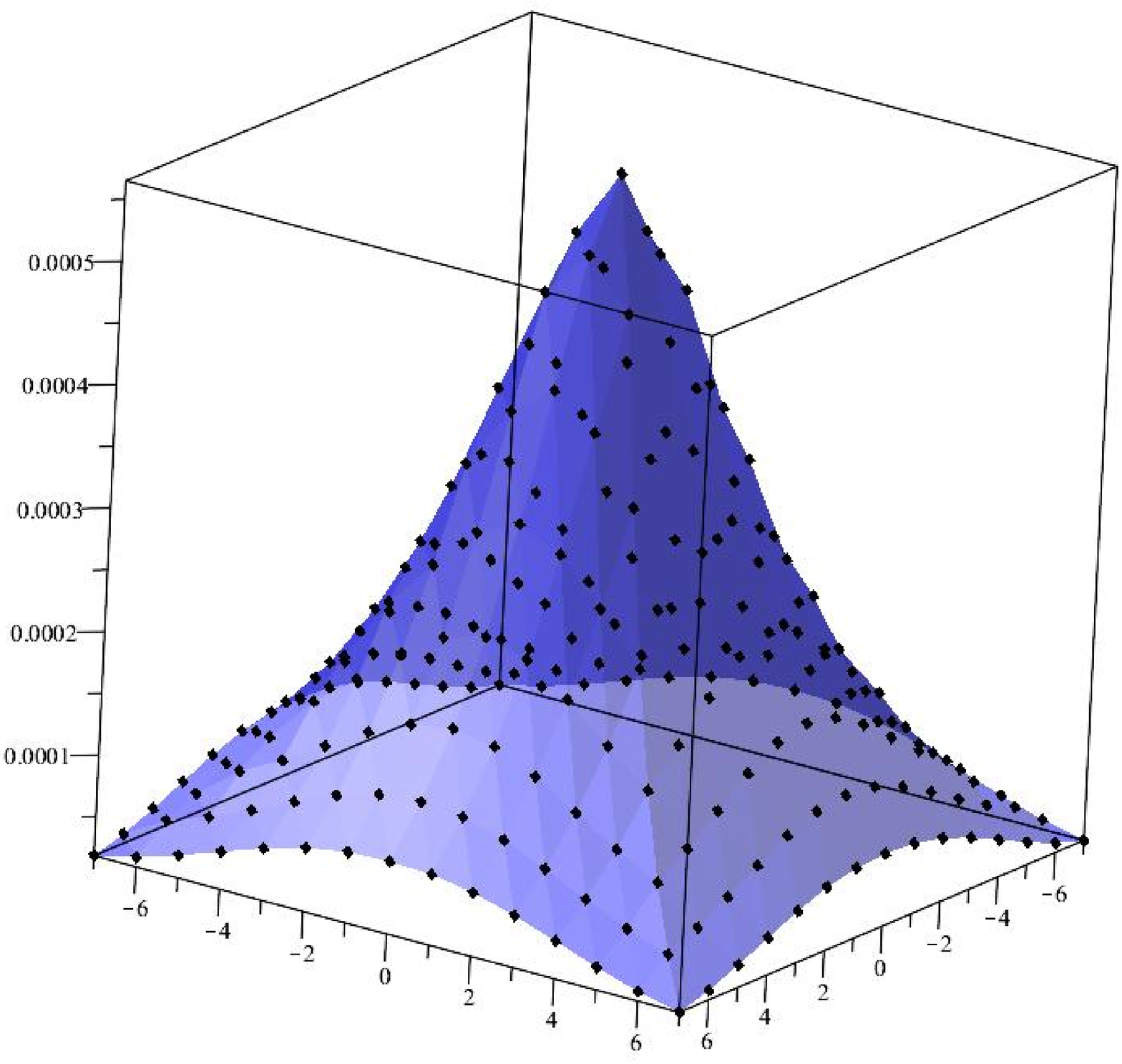}\hspace*{-5mm}
&
\hspace*{-5mm}\includegraphics[width=75mm]{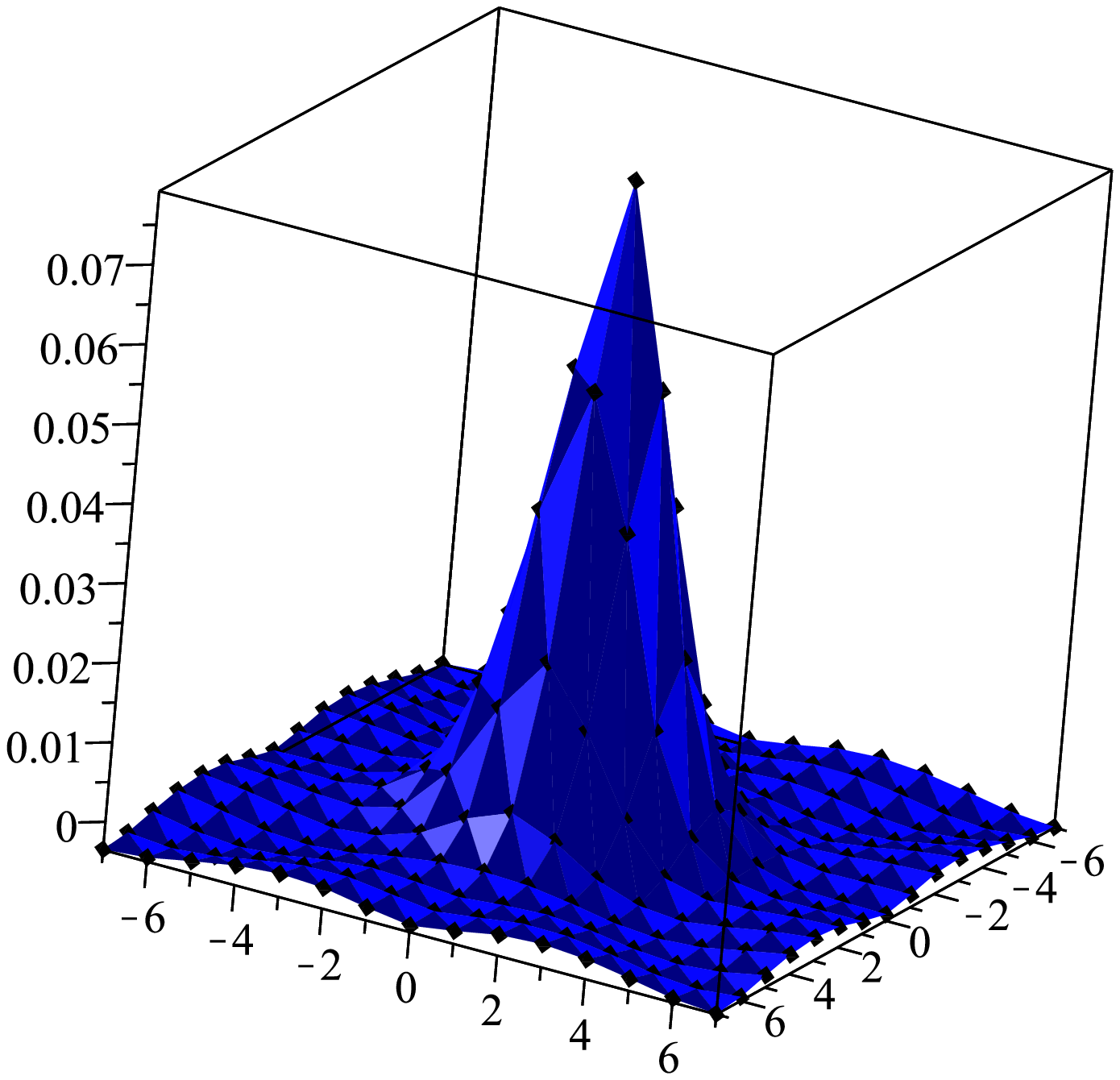}\hspace*{-5mm}
\\
(a) $s=2.5$ & (b) $s=6$
\end{tabular}
\label{krata}
\caption{The graph of the limit solution $F_s$ from
Theorem~\ref{th:super_diffusion} for $s=2.5$ (a) and $s=6$ (b), respectively.}
\label{fig:Fs}
\end{figure}

\begin{remark}
Theorem~\ref{th:super_diffusion} shows that the distribution
spreads proportionally to $t^{\frac{1}{\alpha}}$ where $\alpha$ is as
in that theorem; cf.\ \cite[Remark~6.2]{PartI}.
When $s>4$, one has normal diffusion since the profile spreads
proportionally to $t^{\frac{1}{2}}$ in this case.
When $2<s<4$, however, we observe superdiffusion because then
the spread of the profile is proportional to $t^\kappa$
with $\kappa=\frac{1}{s-2}>\frac{1}{2}$.
In particular, when $s=3$, then the profile spreads linearly in time,
which is a ballistic behaviour.

One can measure the spread, e.g.\ with the full width at half maximum (FWHM),
which, for our purpose, we can define as
\begin{align*}
  \FWHM(t) \defeq 2\sup\biggl\{r>0: \; & u_{x,y}(t) \le \frac{1}{2}u_{0,0}(t) \\
  &\text{for all}\;\; x,y\in\ZZ \;\;\text{with}\;\; |x|^2+|y|^2\ge r^2\biggr\}.
\end{align*}
One can show that $\FWHM(t) \sim ct^{\frac{1}{\alpha}}$ as $t\to\infty$
with some $c>0$; cf.\ \cite[Remark~6.2]{PartI} for the one-dimensional case.
\end{remark}

\begin{remark}\label{re:finite_truncations}
Let us consider finite truncations of the Mellin transformation \eqref{deftLM}
of the $k$-path Laplacian, i.e.\ set
\[
  \LMsN{s}{N} \defeq \sum_{k=1}^N \frac{1}{k^s}L_k
\]
for $N\in\NN$.
By Lemma~\ref{le:lk} this operator is unitarily equivalent to the operator of
multiplication by the function
\[
  \lMsN{s}{N}(p,q) = \sum_{k=1}^N \frac{1}{k^s}\lk(p,q)
\]
where $\lk$ is defined in that lemma.
Using Lemmas~\ref{le:diff_quot} and \ref{le:asymp1} one can show in
a similar way as above that
\[
  \lk(p,q) = \frac{2k^3+k}{3}(p^2+q^2) + \rmO(p^4+q^4), \qquad p,q\to0
\]
and hence
\[
  \lMsN{s}{N}(p,q) = \sum_{k=1}^N \frac{2k^3+k}{3k^s}(p^2+q^2) + \rmO(p^4+q^4),
  \qquad p,q\to0.
\]
This leads to normal diffusion by Lemma~\ref{le:general_asymp}
and not to a superdiffusive process like in the non-truncated Mellin transformation.
However, the diffusion speed and the variance of the limiting normal distribution
grow with $N$, e.g.\
if one measures the former with the full width at half maximum, one gets
\begin{equation}\label{fwhm_finiteN}
  \FWHM(t) \sim 2\Biggl((\ln 2)\sum_{k=1}^N \frac{2k^3+k}{3k^s}\Biggr)^{\frac{1}{2}}
  t^{\frac{1}{2}}, \qquad t\to\infty;
\end{equation}
cf.\ \cite[Remark~6.4]{PartI}.
As $N\to\infty$ one has the following behaviour,
\begin{alignat*}{3}
  \sum_{k=1}^N \frac{(2k^2+1)k}{3k^s} &\sim \frac{2}{3(4-s)}N^{4-s},
  \qquad & & N\to\infty, \qquad & &\text{if}\;\; s\in(2,4),
    \\[1ex]
  \sum_{k=1}^N \frac{(2k^2+1)k}{3k^s} &\to \frac{2\zeta(s-3)+\zeta(s-1)}{3}\,,
  \qquad & & N\to\infty, \qquad & &\text{if}\;\; s\in(4,\infty).
\end{alignat*}
When $s>4$ (i.e.\ when $u$ in Theorem~\ref{th:super_diffusion} shows
normal diffusion), the coefficient in \eqref{fwhm_finiteN} converges to the
corresponding coefficient for $u$ as $N\to\infty$,
and the limiting normal distributions converge to the limiting distribution
from Theorem~\ref{th:super_diffusion}.
On the other hand, when $s\in(2,4)$, the coefficient in \eqref{fwhm_finiteN}
diverges as $N\to\infty$.
So, although one has normal diffusion for every finite $N$, the speed of
the diffusion --- and also the variance of the limiting normal distribution --- can
be made arbitrarily large if $N$ is chosen big enough.
\end{remark}

\end{document}